\documentclass[11pt,letterpaper]{article}
\usepackage[utf8]{inputenc}
\usepackage[margin=1.0in]{geometry}
\usepackage{setspace}
\usepackage{amsmath}
\usepackage{amsthm}
\usepackage{amssymb}
\usepackage{mathrsfs}
\usepackage{graphicx}
\usepackage{dsfont}
\usepackage{enumerate}
\usepackage{mathabx}
\usepackage{mathdots}
\usepackage{authblk}
\usepackage{titlesec}
\usepackage{soul}
\usepackage{tikz-cd}
\usepackage{mathtools}
\usepackage{slashed}

\usepackage{hyperref}
\hypersetup{
    colorlinks=true,
    urlcolor=cyan,
    linkcolor=blue,
    citecolor=blue,
}

\newtheoremstyle{citedtheorem}%
  {3pt}% (space above)
  {3pt}% (space below)
  {\itshape}% (body font)
  {}% (indent amount)
  {\bfseries}% {theorem head font}
  {.}% {punctuation after theorem head}
  {.5em}% {space after theorem head}
  {\thmname{#1} \thmnumber{#2} \thmnote{\normalfont#3}}% {theorem head spec}

\newtheoremstyle{citeddefinition}%
  {3pt}% (space above)
  {3pt}% (space below)
  {\normalfont}% (body font)
  {}% (indent amount)
  {\bfseries}% {theorem head font}
  {.}% {punctuation after theorem head}
  {.5em}% {space after theorem head}
  {\thmname{#1} \thmnumber{#2} \thmnote{\normalfont#3}}% {theorem head spec}

\newtheoremstyle{citedexample}%
  {3pt}% (space above)
  {3pt}% (space below)
  {\normalfont}% (body font)
  {}% (indent amount)
  {\itshape}% {theorem head font}
  {.}% {punctuation after theorem head}
  {.5em}% {space after theorem head}
  {\thmname{#1} \thmnumber{#2} \thmnote{\normalfont#3}}% {theorem head spec}

\theoremstyle{plain}
\newtheorem{theorem}{Theorem}[section]
\newtheorem{lemma}[theorem]{Lemma}
\newtheorem{corollary}[theorem]{Corollary}

\newtheorem{definition}[theorem]{Definition}

\theoremstyle{citedtheorem}

\theoremstyle{definition}

\theoremstyle{citeddefinition}

\theoremstyle{remark}
\newtheorem{remark}[theorem]{Remark}

\theoremstyle{citedexample}

\usepackage{cleveref}

\newcommand{\cE}{\mathcal{E}}

\newcommand{\cL}{\mathcal{L}}

\newcommand{\fF}{\mathfrak{F}}

\newcommand{\fg}{\mathfrak{g}}

\newcommand{\fu}{\mathfrak{u}}

\newcommand{\rD}{\mathrm{D}}

\newcommand{\sD}{\mathscr{D}}
\newcommand{\sE}{\mathscr{E}}

\newcommand{\sG}{\mathscr{G}}

\newcommand{\sS}{\mathscr{S}}

\renewcommand{\a}{\alpha} % \a gave ???
\renewcommand{\b}{\beta} % \b gave bars under letters

\newcommand{\g}{\gamma}

\newcommand{\n}{\nu}
\newcommand{\x}{\xi}

\newcommand{\vp}{\varphi}

\renewcommand{\l}{\ell} % \l gave produces Polish l

\newcommand{\N}{\mathbb{N}}
\newcommand{\Z}{\mathbb{Z}}

\newcommand{\R}{\mathbb{R}}
\newcommand{\C}{\mathbb{C}}
\renewcommand{\d}{\partial} % \d already defined, produces ???

\newcommand{\sub}{\subseteq}

\newcommand{\8}{\infty}
\newcommand{\<}{\left\langle}
\renewcommand{\>}{\right\rangle}

\DeclareMathOperator{\ad}{ad}

\DeclareMathOperator{\End}{End}
\DeclareMathOperator{\id}{id}

\DeclareMathOperator{\rank}{rank}
\DeclareMathOperator{\Tr}{Tr}

\DeclareMathOperator{\Cl}{C\l}

\DeclareMathOperator{\Spec}{Spec}

\DeclareMathOperator{\supp}{supp}

\DeclareMathOperator{\Dir}{\slashed{\rD}}

\DeclareMathOperator{\Spin}{Spin}

\title{A Uniqueness Result for the Calder\'{o}n Problem for $U(N)$-connections coupled to spinors}
\author{Carlos Valero}

\begin{document}

\maketitle

\begin{abstract}
    In this paper we define a Dirichlet-to-Neumann map for a twisted Dirac Laplacian acting on bundle-valued spinors over a spin manifold. We show that this map is a pseudodifferential operator of order $1$ whose symbol determines the Taylor series of the metric and connection at the boundary. We go on to show that if two real-analytic connections couple to a spinor via the Yang--Mills--Dirac equations with appropriate boundary conditions, and have equal Dirichlet-to-Neumann maps, then the two connections are locally gauge equivalent. In the abelian case, the connections are globally gauge equivalent.
\end{abstract}

\tableofcontents

\newpage

\section{Introduction}\label{sec1}

In this paper, we consider a Calder\'{o}n inverse problem for unitary connections on Hermitian vector bundles over spin manifolds that couple to spinor fields via the Yang--Mills--Dirac system. In particular, we consider the Dirichlet-to-Neumann map for the twisted Dirac Laplacian acting on vector-valued spinors, and investigate how the introduction of the spin structure affects the recovery of the metric and connection from boundary data.

The Calder\'{o}n problem has its origin in the physical question of whether one can determine the conductivity of a medium by making measurements on the boundary of potential functions and the induced currents. Geometrically, this corresponds to the question of whether one can determine the metric on a manifold with boundary, up to isometry, from knowledge of its Dirichlet-to-Neumann map, which sends a function on the boundary to the normal derivative of its harmonic extension. Much work has since been done on the Calder\'{o}n problem for the scalar Laplacian; we refer the reader to \cite{Uhlmann2009}, or \S 1 of the more recent \cite{DKN2018}, for a survey of uniqueness results in the literature. 

Many natural extensions of this problem arise when one considers other important second-order elliptic operators induced by some geometric structure. For example, one may fix a vector bundle $E$ over a Riemannian manifold $(M,g)$ with boundary, and consider a connection $\nabla$ on this vector bundle. One may then ask to what extent the Dirichlet-to-Neumann map for the connection Laplacian $\nabla^* \nabla$ determines the connection, up to gauge equivalence. This problem has been studied in a few recent works, in which a number of uniqueness results are proved. In \cite{Kurylev2018}, a Riemannian manifold, Hermitian vector bundle, and connection are reconstructed from the hyperbolic Dirichlet-to-Neumann map associated to the wave equation of the connection Laplacian. In \cite{Ceki2017}, the elliptic Dirichlet-to-Neumann map is considered, and it is shown using methods of complex geometrical optics that the Dirichlet-to-Neumann map for a connection Laplacian determines the connection up to gauge for a class of vector bundles over special Riemannian manifolds, namely conformally transversally anisotropic manifolds with injective ray transform. Using new methods of geometric analysis and Runge approximation, Ceki\'{c} shows in \cite{Ceki2020} that a Hermitian vector bundle and Yang--Mills connection can be recovered, up to gauge transformations, from the Dirichlet-to-Neumann map of its connection Laplacian. Finally, in the recent preprint \cite{Gabdurakhmanov2021}, the authors reconstruct a Euclidean vector bundle and connection from the connection Laplacian Dirichlet-to-Neumann map when all of the data is real-analytic and the dimension of $M$ is at least $3$.

%%%%%%%%%%%%%%%%%%%%%%%%%%%%%%%%%%%%%%%%%%%%%%%%%%%%%%%%%%%%%%%%%%%%%%%%%%%%%%%%%

On the other hand, inverse boundary problems for first-order Dirac operators have also been studied in the literature. One important uniqueness result is due to Kurylev and Lassas \cite{Kurylev2009} who showed that a Riemannian manifold and super-vector bundle can be recovered from the spectrum and eigenfunctions of the corresponding Dirac operator on the boundary. There are also many works that consider the problem of recovering magnetic potentials from boundary data corresponding to first-order Dirac operators. Much in the spirit of the present paper, Salo and Tzou \cite{Salo2008, Salo2010} have shown, using the method of limiting Carleman weights, that a potential and magnetic field can be recovered from the Cauchy data of an associated Dirac equation over a compact domain in $\R^n$. In the terminology used in the present paper, the magnetic field corresponds to the curvature of an abelian connection.

In this paper, we consider the question of determining a $U(N)$-connection $A$ up to gauge equivalence, from the Dirichlet-to-Neumann map of the twisted Dirac Laplacian $\Dir_A^2$. More precisely, we consider a compact spin manifold $M$ with boundary $\d M$, and a Hermitian vector bundle $E$ over $M$. Then the Dirac operator $\Dir$ of $M$ can be defined, acting on its bundle of complex spinors $S$. Now given any connection $A$ on $E$, we may endow the bundle $S \otimes E$ with the connection $\omega^{\mathrm{s}} \otimes A$, where $\omega^{\mathrm{s}}$ is the spin connection on $S$, induced by the Levi-Civita connection. With this connection, we may define a twisted, or covariant, Dirac operator $\Dir_A$, which like $\Dir$ is a first-order, elliptic, and self-adjoint operator acting on sections of $S \otimes E$.

In order for the Dirichlet problem to be well-defined, we consider the square of this operator, $\Dir_A^2$. The Dirichlet-to-Neumann map associated to $\Dir_A^2$ can be thus defined by sending any section $\chi$ of $S \otimes E|_{\d M}$ to the covariant normal derivative of its harmonic extension with respect to $\Dir_A^2$. It is straightforward to generalize this to include non-zero mass terms and zeroth-order potentials, provided that the Dirichlet-problem is well-defined. We summarize with the following definition:

%%%%%%%%%%%%%%%%%
\begin{definition}\label{DN map for spinors}
Let $M$ be an $n$-dimensional compact spin manifold with boundary $\d M$, and let $g$ be a Riemannian metric on $M$. Let $S$ be the spinor bundle associated to some fixed spin structure on $M$, and let $E$ be a Hermitian bundle of rank $N$ on $M$. Let $A$ denote a $U(N)$-connection on $E$, and let $Z$ be an endomorphism of $S \otimes E$. For any $m \in \R$ such that $m^2 \notin \Spec{\left( \Dir_A^2 + Z \right)}$, we define the \em{Dirichlet-to-Neumann map}
\begin{equation}
    \Lambda_{g, A, Z, m} : C^\8\left( S \otimes E|_{\d M} \right) \to C^\8\left( S \otimes E|_{\d M} \right)
\end{equation}
as follows. For $\chi \in C^\8\left( S \otimes E|_{\d M} \right)$, we may solve the Dirichlet problem
\begin{equation}\label{Dirichlet problem for chi}
    \begin{cases}
    \Dir_A^2 \vp + Z \vp - m^2 \vp = 0, \\
    \vp|_{\d M} = \chi,
    \end{cases}
\end{equation}
to obtain a unique solution $\vp \in C^\8(S \otimes E)$. We then define $\Lambda_{g,A,Z,m}(\chi) := \nabla^A_\n \vp \big|_{\d M}$ where $\n$ is the inward unit normal to $\d M$.
\end{definition}

We will often suppress subscripts on the Dirichlet-to-Neumann map that are understood to be fixed, and indicate only the relevant ones. For example, when we recover the connection with the background metric, mass parameter, and endomorphism fixed in Section \ref{section: recovering connections}, the Dirichlet-to-Neumann map is simply denoted $\Lambda_A$.

\begin{remark}\label{extension of DN map to sobolev}
The Dirichlet-to-Neumann map can be extended to a map
\begin{equation}
    \Lambda_{g, A, Z, m} : H^{\frac{1}{2}}\left( S \otimes E|_{\d M} \right) \to H^{-\frac{1}{2}}\left( S \otimes E |_{\d M} \right),
\end{equation}
where for any vector bundle $\cE$ and $s \in \R$, $H^s(\cE)$ denotes the Hilbert space of $s$-Sobolev sections of $\cE$. That is, $H^s(\cE)$ denotes the space of distributional sections of $\cE$ that are represented by a tuple of $H^s$ functions in any smooth local trivialization. Indeed, there is a natural weak formulation of Definition \ref{DN map for spinors} if we introduce a modified Dirichlet-to-Neumann map $\hat{\Lambda}_{g,A,Z,m}$ as follows: for any $\chi \in H^{\frac{1}{2}}\left(S \otimes E|_{\d M}\right)$, we can again solve \eqref{Dirichlet problem for chi} to obtain a unique $\vp \in H^{1}(S \otimes E)$. We then define $\hat{\Lambda}_{g,A,Z,m}(\chi) \in H^{-\frac{1}{2}}(S \otimes E|_{\d M})$ by the property that
\begin{equation}
    \< \hat{\Lambda}_{g,A,Z,m}(\chi), \zeta \> = \int_M \< \Dir_A \vp, \Dir_A \psi \> - \int_M \< \left( Z - m^2 \right) \vp, \psi  \>
\end{equation}
holds for all $\zeta \in H^{\frac{1}{2}}(S \otimes E|_{\d M})$, where $\psi \in H^{1}(S \otimes E)$ is any extension of $\zeta$, which exists by the standard trace theorems for Sobolev spaces \cite[\S 11]{booss1993}. In the smooth case, the Green's formula for the Dirac operator \cite[Equation 5.7]{LawsonMichelsohn}
\begin{equation}
    \int_M \< \Dir_A \sigma_1, \sigma_2 \> = \int_M \< \sigma_1, \Dir_A \sigma_2 \> -\int_{\d M} \< \nu \cdot \sigma_1, \sigma_2  \>
\end{equation}
implies that $\hat{\Lambda}_{g,A,Z,m}(\chi) = -\gamma(\nu) \Dir_A \vp |_{\d M}$, where $\gamma(\nu)$ denotes Clifford multiplication by $\nu$, and so differs from $\Lambda_{g,A,Z,m}(\chi)$ as given in Definition \ref{DN map for spinors} by tangential derivatives. In particular, we may say that $\hat{\Lambda}_{g,A,Z,m}$ and $\Lambda_{g,A,Z,m}$ contain the same information about the geometric data, such as the metric and connection. In this paper, we thus are free to restrict ourselves to considering $\Lambda_{g,A,Z,m}$. 
\end{remark}

We note that there are some natural gauge invariances that arise from Definition \ref{DN map for spinors}, which shall be explored in greater detail in Section \ref{section: preliminaries}. The first is the gauge-invariance of the connection $A$. Recall that two connections $A$ and $A'$ are called gauge equivalent if there exists a unitary automorphism $G$ of $E$, otherwise called a gauge transformation, such that their covariant derivatives are related by
\begin{equation}\label{A' A gauge equivalent}
    \nabla^{A'} = G^{-1} \circ \nabla^A \circ G.
\end{equation}
We say that $A$ and $A'$ are locally gauge equivalent about a point $x \in M$ if there is an open neighbourhood $U$ of $x$ such that the restrictions of $A$ and $A'$ to $U$ are gauge equivalent.

It is easy to see that if there exists a $G$ as in \eqref{A' A gauge equivalent} with $G|_{\d M} = \id$, then $\Lambda_{A'} = \Lambda_{A}$. Indeed, if $\vp$ is the solution to \eqref{Dirichlet problem for chi} for $A$, then $\left(\id_S \otimes G \right)^{-1} \vp$ is the solution to \eqref{Dirichlet problem for chi} with $A$ replaced by $A'$. Therefore, suppressing $\id_S$, we have
\begin{equation}
    \Lambda_{A'}(\chi) = \nabla_{\n}^{A'} \left( G^{-1} \vp \right) \big|_{\d M} = G^{-1} \nabla_{\n}^A \vp \big|_{\d M} = \Lambda_{A}(\chi).
\end{equation}
Thus, given $\Lambda_{A} = \Lambda_{A'}$, we can only every recover the connection up to a gauge transformation that is equal to the identity on the boundary.

Going further, we would like to say that the definition \ref{DN map for spinors}, inasmuch as it depends on the geometry of the metric $g$, depends only on the isometry class of $g$. However, since spin structures, and hence spinor bundles, are defined with respect to a fixed metric, we must take care in relating the Dirichlet-to-Neumann maps of two different metrics. In Section \ref{section: preliminaries}, we explain how a diffeomorphism $\Phi : M \to M$ induces an isomorphism of associated spinor bundles $\tilde{\Phi} : S_{\Phi^*g} \to S_g$. Then it is easy to prove that the Dirichlet-to-Neumann map is diffeomorphism-invariant:

\begin{lemma}\label{diffeomorphism invariance of DN map}
Let $\Phi : M \to M$ be a diffeomorphism such that $\Phi|_{\d M} = \id$. Then 
\begin{equation}\label{equality of DN maps for isometric metrics}
    \Lambda_{\Phi^*g} = \tilde{\Phi}^{-1}|_{\d M} \circ \Lambda_g \circ \tilde{\Phi}|_{\d M}.
\end{equation}
\end{lemma}

In this paper, we want to consider a Calder\'{o}n Problem for a $U(N)$-connection $A$, which couples to an $E$-valued spinor $\phi$ through some natural equations arising from physics. To this end, recall that given a connection $A$, one may define its curvature $F_A$. If $P$ denotes the principal $U(N)$-bundle of unitary frames of $E$, then $F_A$ is an $\ad{P}$-valued $2$-form. In physics, the curvature of a connection corresponds to a force field; if the connection is abelian then its curvature is the electromagnetic field. Now let $m \in \R$ be such that $m^2$ is not in the Dirichlet spectrum of $\Dir_A^2$. We then assume that there exists an $E$-valued spinor $\phi$ such that $(A, \phi)$ satisfies the following second-order Yang--Mills--Dirac system:
\begin{equation}
    \begin{cases}
            \Dir_A^2 \phi = m^2 \phi, \label{Yang mills dirac spinor 2nd order} \\
            d_A^* F_A = J(\phi),
    \end{cases}
\end{equation}
where the current $J(\phi)$ is an $\ad{P}$-valued $1$-form depending quadratically o $\phi$. The system \eqref{Yang mills dirac spinor 2nd order} arises when a gauge field interacts with fermionic matter, as represented by the spinor field $\phi$. In this paper we prove the following result:

\begin{theorem}\label{main theorem}
Let $(M,g)$ be a compact $n$-dimensional real-analytic spin manifold with boundary, and let $E$ be a real-analytic Hermitian vector bundle over $M$. Let $A$ and $B$ be $U(N)$-connections on $E$, real-analytic in the interior of $M$, and let $\phi$ and $\psi$ be $E$-valued spinors on $M$ such that $(A, \vp)$ and $(B,\psi)$ both satisfy the second-order Yang--Mills--Dirac system \eqref{Yang mills dirac spinor 2nd order}, and such that $\vp|_{\d M} = \psi|_{\d M}$. If $\Lambda_{A,m} = \Lambda_{B,m}$, then $A$ and $B$ are locally gauge equivalent about every point in $M$. If $N = 1$, then $A$ and $B$ are globally gauge equivalent.
\end{theorem}

In other words, if a connection $A$ couples to an auxiliary $E$-valued spinor via the second-order Yang--Mills--Dirac system, then by making boundary measurements $\Lambda_A(\chi)$ of other spinor fields, one can determine the curvature of the connection at any point up to conjugation. In the abelian case, one can determine the connection up to gauge.

\section{Preliminaries}\label{section: preliminaries}

In this section, we recall some of the basic tools and constructions that we need to state the central problem and prove the main result. We review the construction of spinor bundles and the Dirac operator in Section \ref{subsection: spin geo}, and extend this to bundle-valued spinors in Section \ref{subsection: connections}. Lastly, we review a few key results from the theory of pseudodifferential operators in \ref{subsection: pseudos}.

\subsection{Spinors and the Dirac Operator}\label{subsection: spin geo}

In this subsection, we briefly recall some basic notions of spin geometry, such as Clifford algebras, Spin manifolds, the construction of the spinor bundle on a Spin manifold, and the definition and fundamental properties of Dirac operators. The important results for our purposes are the local formulae \eqref{spin connection local} for the spin connection, and \eqref{Dirac operator defn} for the Dirac operator. For details on the results presented in this section, we refer the reader to \cite{LawsonMichelsohn}.

\begin{definition}\label{Clifford algebra defn}
The {\em Clifford algebra} $\Cl_n$ is the real associative algebra generated by vectors in $\R^n$, with a product $\cdot$ satisfying the relation
\begin{equation}\label{clifford algebra relation}
   v \cdot w + w \cdot v = -2 \< v,w \> 
\end{equation}
for all $v,w \in \R^n$, where $\< \cdot, \cdot \>$ denotes the standard inner product on $\R^n$. 
\end{definition}

Note that we have a natural inclusion $\R^n \sub \Cl_n$, which can be extended to a vector space isomorphism $\Lambda^*\R^n \to \Cl_n$ by fixing an orthonormal basis $\{ e_i \}_{i=1}^n$ of $\R^n$ and sending
\begin{equation}
    e_{i_1} \wedge \cdots \wedge e_{i_k} \mapsto e_1 \cdots e_k.
\end{equation}
In particular, we have $\dim{\Cl_n} = 2^n$. 

Clifford algebras play an important role throughout geometry and physics. One reason for this is that $\Cl_n$ naturally contains the spin group $\Spin(n)$, which can be defined as a double cover of the rotation group $SO(n)$. That is, there exists a surjective Lie group homomorphism $\lambda : \Spin(n) \to SO(n)$ such that $\ker{\lambda} = \{1,-1\}$. Thus, $\lambda$ is a 2-sheeted covering map.

\begin{remark}
It follows from the preceding that $\Spin(n)$ is in fact the universal cover of $SO(n)$ for $n \geq 3$, since $\pi_1(SO(n)) = \Z_2$ for $n \geq 3$.
\end{remark}

\begin{definition}
Let $(M,g)$ be an oriented $n$-dimensional Riemannian manifold {\em without} boundary. A {\em spin structure} $P_{\Spin}$ on $M$ is a principal $\Spin(n)$-bundle, which covers the $SO(n)$-bundle of orthonormal frames of $M$, in such a way that the covering map is compatible with the $2$-fold covering map $\lambda : \Spin(n) \to SO(n)$. When such a bundle $P_{\Spin}$ exists, we say $M$ is {\em spin}.
\end{definition}

\begin{remark}
Not every oriented Riemannian manifold is spin. In fact, the obstruction to being spin is entirely contained in the second Stiefel-Whitney class of $M$. Even when $M$ is spin, the spin structure need not be unique. See \cite{LawsonMichelsohn} for details. 
\end{remark}

A spin manifold is endowed with a distinguished complex vector bundle $S$ called the {\em spinor bundle}. It is defined as follows. First, fix a spin structure $P_{\Spin}$ on $M$. Let $\rho : \Cl_n \to \End{\C^k}$ be an irreducible $\Cl_n$-module. It follows from the structure of $\Cl_n$ that $k = \left[ \frac{n}{2} \right]$. Then $\rho$ restricts to a representation of $\Spin(n)$, called the {\em spinor representation}. The spinor bundle $S$ is then the bundle associated to $P_{\Spin}$ and the spinor representation, $S := P_{\Spin} \times_{\rho} \C^k$. The sections of $S$ are called (complex) {\em spinors}. One can also define real spinors but they shall not concern us here.

\begin{remark}
The spinor representations have the property that $-1 \in \Spin(n)$ acts non-trivially. Since $\lambda(-1) = 1$, where $\lambda : \Spin(n) \to SO(n)$ is the $2$-fold covering map described above, it follows that the spinor representations do not descend to representations of $SO(n)$. Therefore, the spinor representations are in some sense the simplest representations that do not correspond to representations of $SO(n)$.  
\end{remark}

Since the typical fibre of $S$ is a $\Cl_n$-module, it follows that $S$ is bundle of modules over $\Cl(M)$, the Clifford bundle of $M$, which is defined by $\Cl(M)_x := \Cl(T_xM)$. In particular, since $TM \sub \Cl(M)$, we have a map $\gamma : TM \to \End{S}$, called {\em Clifford multiplication}. It is possible to endow $S$ with a Hermitian metric such that $\g(e)$ is skew-symmetric for all unit vectors $e \in TM$. We assume $S$ has such a metric henceforward.

So far we have presented these definitions for a spin manifold without boundary. A spin manifold with boundary $M$ is defined to be any closed domain of a spin manifold $N$, whose spin structure is the restriction of the spin structure on $N$. All of the above definitions then extend to this setting.

Now, the Levi-Civita connection $\omega$ on the oriented orthonormal frame bundle of $M$ lifts to a connection $\omega^{\mathrm{s}}$ on any given spin structure, which we call the {\em spin connection}. This connection induces a covariant derivative $\nabla^{\mathrm{s}}$ on sections of $S$, which can be explicitly described with respect to a local trivialization as follows.

Let $(e_i)_i$ be a local orthonormal frame for $M$, and let $\omega^i_{\ j}$ be the Levi-Civita connection $1$-form with respect to this frame, so that
\begin{equation}
    \nabla_X e_j = \omega^i_{\ j}(X) e_i.
\end{equation}
This frame can be lifted to a local section of the spin structure $P_{\Spin}$, which we can regard as a local frame $(\sigma_\a)_\a$ for $S$ (although there is another lifted frame, namely $(-\sigma_\a)_\a$, the choice of lift is immaterial here). Then, with respect to this frame of spinors, the spin connection takes the form
\begin{equation}\label{spin connection local}
    \nabla^\mathrm{s} \sigma_\a = -\frac{1}{2} \sum_{i < j} \omega^{i}_{\ j} \otimes \g(e_i) \g(e_j) \sigma_\a.
\end{equation}

The structure of a Clifford module allows us to define the Dirac operator $\Dir$ on sections of $S$ as follows. For $\vp \in \Gamma(S)$, let $(e_i)_i$ be any orthonormal frame in an open set $U$. Then, in $U$,
\begin{equation}\label{Dirac operator defn}
     \Dir \vp :=  \sum_{i=1}^n \gamma(e_i) \nabla^s_{e_i} \vp.
\end{equation}
This definition does not depend on the choice of orthonormal frame. The Dirac operator plays an crucial role in physics, where it occurs in the equations of motion for spinor fields, which represent fermionic matter. The square of the Dirac operator satisfies the famous Lichnerowicz formula,
\begin{equation}\label{lichnerowicz}
    \Dir^2 \vp = \left( \nabla^\mathrm{s} \right)^* \nabla^{\mathrm{s}} \vp + \frac{1}{4} R \vp
\end{equation}
where $\left( \nabla^s \right)^*$ denotes the formal adjoint of $\nabla^A$ with respect to the $L^2$-inner product, and $R$ is the scalar curvature of $g$. Thus one can transfer questions about the Dirac Laplacian $\Dir^2$ to questions about the spin connection Laplacian at the expense of a curvature term.

\subsection{$U(N)$-connections and bundle-valued spinors}\label{subsection: connections}

We now want to introduce an auxiliary Hermitian vector bundle $(E,h)$ equipped with a connection $A$, whose curvature $F_A$ corresponds to some physical force. Moreover, we want to introduce a mechanism to couple this connection to spinor fields, corresponding to the physical interaction between the force field $F_A$ and fermions. If the connection is abelian, then its curvature is the familiar electromagnetic field.

Let $(E,h)$ be a Hermitian vector bundle of rank $N$, which is associated to its bundle of complex orthonormal frames $P$, which is a principal $U(N)$-bundle. Consider the bundle $S \otimes E$, the bundle of $E$-valued spinors, which is associated to the principal bundle $P_{\Spin} \times_M P$. Given a $U(N)$-connection $A$ on $P$, we can equip $S \otimes E$ with the connection $\omega^{\text{s}} \otimes A$. The corresponding covariant derivative on sections of $S \otimes E$ is denoted by $\nabla^A$. Thus, with respect to a local trivialization of $E$, an $E$-valued spinor $\psi$ is given by a tuple of $N$ complex spinors, and 
\begin{equation}\label{covariant derivative of spinor}
    \left( \nabla_X^A \psi \right)^a = \nabla_X^\mathrm{s} \psi^a + A^a_{\ b}(X) \psi^b \ \ \ \ a \in \{1, \dots, N\},
\end{equation}
where $A^a_{\ b}$ are the components of the $\fu(N)$-valued $1$-form representing the connection $A$ in this local trivialization. We can also define a twisted Dirac operator acting on $\psi \in \Gamma(S \otimes E)$ by
\begin{equation}\label{twisted dirac operator}
    \Dir_A \psi := \sum_{i=1}^n \gamma(e_i) \nabla^A_{e_i} \psi
\end{equation}
where $(e_i)_i$ is any local orthonormal frame on $M$. Note that Clifford multiplication on $E$-valued spinors acts on the $S$ factor; that is, we have identified $\gamma(e_i)$ with $\gamma(e_i) \otimes \id$. With respect to a local trivialization of $E$, we can view $\psi \in \Gamma(S \otimes E)$ as $N$ complex spinors, and the twisted Dirac operator takes the form
\begin{equation}\label{twisted dirac operator local trivialization}
    \left( \Dir_A \psi \right)^a = \sum_{i=1}^n \gamma(e_i) \nabla^s_{e_i} \psi^a + A^{a}_{\ b}(e_i) \gamma(e_i) \psi^b, \ \ \ \ a \in \{1,\dots, N\}. 
\end{equation}
The twisted Dirac operator also satisfies a twisted Lichnerowicz formula,
\begin{equation}\label{twisted Lichnerowicz formula}
    \Dir_A^2 \psi = \left( \nabla^A \right)^* \nabla^A \psi + \frac{1}{4} R \psi + \fF_A \cdot \psi,
\end{equation}
where the curvature operator $\fF_A : S \otimes E \to S \otimes E$ is defined by
\begin{equation}\label{curvature operator definition}
    \fF_A (\sigma \otimes \eta) := \frac{1}{2} \sum_{j,k=1}^n \left(\gamma(e_j) \gamma(e_k) \sigma \right) \otimes \left( F_A(e_j, e_k) \eta \right).
\end{equation}

A gauge transformation is a section of $U(E)$, the bundle of unitary automorphisms of $E$. These sections form a group called the gauge group, which we denote $\sG(E)$. A gauge transformation $G \in \sG(E)$ acts on a connection $A$ by taking it to the connection $A'$ defined by 
\begin{equation}
    \nabla^{A'}  = G^{-1} \circ \nabla^A \circ G.
\end{equation}
Two connections $A$ and $A'$ are considered to be gauge equivalent if they lie in the same $\sG(E)$-orbit. The notion of local gauge equivalence over an open set is similarly defined using the restricted bundles and connections. If $A'$ is related to $A$ by a local gauge transformation $G$, then with respect to a local trivialization, the connection $1$-forms representing $A$ and $A'$ are related by
\begin{equation}
    A' = G^{-1} A G + G^{-1} dG.
\end{equation}
Note that $G^{-1}dG$ is indeed a $\fu(n)$-valued $1$-form.

\subsection{Pseudodifferential Calculus}\label{subsection: pseudos}

The proof of Theorem \ref{main theorem} requires a few key results from the theory of pseudodifferential operators, which we recall briefly here. For details and proofs, we refer the reader to \cite{Treves1980}, for example.

Let $W \subseteq \R^n$ be open. Then the symbol class $S^m(W)$ is defined to be the space of all functions $p \in C^\8(W \times \R^n)$ satisfying for all $\a,\b \in \N^{n}$,
\begin{equation}
    \left| \d_\x^\a \d_{x}^\b p(x,\x) \right| \leq C_{\a, \b} \< \x \>^{m - |\a|},
\end{equation}
where $\< \x \> := \sqrt{1 + |\x|^2}$ is the usual regularization of $|\x|$. The symbol class $S^m(W, \C^{k \times k})$ is then the space of all matrix-valued functions whose entries are in $S^m(W)$. Each $p \in S^m(W, \C^{k \times k})$ yields a map $P : C^\8_c(W, \C^k) \to C^\8(W, \C^k)$ given by
\begin{equation}\label{pseudodifferential definition}
    (P w)(x) := \int e^{i x \cdot \x} p(x,\x) \widehat{w}(\x)\, d\x,
\end{equation}
where $\widehat{w}$ denotes the Fourier transform of $w$. We say that $P \in \Psi^{m}(W, \C^{k})$ if $P$ has the form \eqref{pseudodifferential definition} for a symbol $p \in S^m(W, \C^{k \times k})$. A pseudodifferential operator is called classical if its symbol is given by an asymptotic series of the form
\begin{equation}\label{classical symbol}
    p(x,\x) = \sum_{j \geq 1} p_{m_j}(x,\x)
\end{equation}
where each term $p_{m_j}$ is positive-homogeneous of degree $m_j$ in $\x$, and the sequence of real numbers $m_j$ is decreasing to $- \infty$.

Now let $E$ be a vector bundle of rank $k$ over a smooth manifold $M$, and let $\sD'(E)$ be the space of $E$-valued distributions on $M$. A map $P : C^\8_0(M,E) \to C^\8(M,E)$ is called a pseudodifferential operator of order $M$ if for every chart $W$ of $M$ and every local trivialization of $E$ over $W$, the induced map is in $\Psi^m(W, \C^{k})$. The space of pseudodifferential operators on $E$ of order $m$ is denoted $\Psi^m(M,E)$. If $P \in \Psi^m(M,E)$ for all $m \in \R$, then we call $P$ a smoothing operator. The space of smoothing operators is denoted $\Psi^{- \infty}(M,E)$. Moreover, if $P \in \Psi^m(M,E)$, then it extends to a map $\sE'(E) \to \sD'(E)$, and if $M$ is compact, then $P$ extends to a map $H^{s}(E) \to H^{s - m}(E)$.

We often work with pseudodifferential operators modulo $\Psi^{-\infty}(M,E)$, since then composition of pseudodifferential operators becomes well-defined. Thus, we will treat two pseudodifferential operators as equivalent if their difference is a smoothing operator. Each equivalence class then corresponds to a symbol modulo $S^{-\infty}$, the intersection of all symbol classes $S^m$. Note that in particular two classical pseudodifferential operators differ by a smoothing operator if and only if their formal symbols are equal modulo $S^{-\infty}$.

if $P_i \in \Psi^{m_i}(M, E)$ for $i =1,2$, then their composition $Q := P_1 \circ P_2$ is well-defined modulo smoothing operators, and its symbol $q$ modulo $S^{-\infty}$ is given by
\begin{equation}
    q(x,\x) \sim \sum_{\a} \frac{1}{\a !} \d_\x^\a p_1(x,\x) D_x^\a p_2(x,\x).
\end{equation}

\section{Boundary determination}\label{section: boundary determination}

In this section, we prove that the Dirichlet-to-Neumann map $\Lambda_{g,A,Z,m}$ is a pseudodifferential operator of order $1$ whose symbol with respect to a local trivialization determines the Taylor series of $g$, $Q$, and $A$ at the boundary, when $A$ is in an appropriate gauge. We shall see upon applying the recipe of Lee and Uhlmann that unlike the analogous proofs for the scalar Laplacian \cite{Lee1989} or the connection Laplacian \cite{Ceki2020}, we need not place any restrictions on the metric. In particular, we shall see that because the connection coefficients of $\omega^s \otimes A$ involve the Levi--Civita connection, and thus the derivatives of the metric, the Taylor series of the metric can be recovered without fixing a representative in its conformal class as is done in \cite{Ceki2020}. 

% \begin{remark}
% This phenomenon is related to the lack of conformal covariance of the Dirac Laplacian.
% \end{remark}
%%%%%% elaborate and cite the no conformal powers thing

\begin{theorem}\label{boundary determination}
The Dirichlet-to-Neumann map in Definition \ref{DN map for spinors} is an elliptic pseudodifferential operator of order $1$. Moreover, in any local trivialization where $A_n = 0$, the total symbol of $\Lambda_{g,A,Z,m}$ determines the Taylor series of $g$, $A$, and $Z$ at the boundary.
\end{theorem}

It is worth stating once the precise meaning of Theorem \ref{boundary determination} in the context of spinor bundles, though we immediately drop this level of precision in the interest of clarity. 

Suppose we have a Riemannian metric $g$ on $M$ inducing a spinor bundle $S$ for a fixed choice of spin structure. Then for a connection $A$ on $E$, and an endomorphism $Z$ of $S \otimes E$, we can construct the corresponding Dirichlet-to-Neumann map $\Lambda_{g,A,Z,m}$ on $S|_{\d M}$. Suppose we have also another Riemannian metric $\Tilde{g}$ on $M$, which induces another spinor bundle $\Tilde{S}$ corresponding to a choice of spin structure. Then for any connection $\Tilde{A}$ on $E$, and endomorphism $\Tilde{Z}$ of $\Tilde{S} \times E$, we can construct the corresponding Dirichlet-to-Neumann map $\Lambda_{\Tilde{g}, \Tilde{A}, \Tilde{Z}, m}$ on $\Tilde{S}|_{\d M}$. To say that these two Dirichlet-to-Neumann maps are equal is of course to mean that they are equal up to isomorphism, but we must take care to specify what kind of isomorphism. 

\begin{definition}\label{equality of DN maps}
We say that the Dirichlet-to-Neumann maps $\Lambda_{g,A,Z,m}$ and $\Lambda_{\Tilde{g}, \Tilde{A}, \Tilde{Z}, m}$ corresponding to data $(g,A,Z)$ and $(\Tilde{g}, \Tilde{A}, \Tilde{Z})$ respectively are isomorphic if there exists an isomorphism $\Phi : S|_{\d M} \to \Tilde{S}|_{\d M}$ induced by an isomorphism of the $SO(n)$-structures of $g$ and $\Tilde{g}$, restricted to $\d M$, such that
\begin{equation}
    \Lambda_{\tilde{g}, \tilde{A}, \tilde{Z}, m} = \Phi \circ \Lambda_{g,A,Z,m} \circ \Phi^{-1}.
\end{equation}
This is equivalent to saying that the local matrix representations of $\Lambda_{g,A,Z,m}$ and $\Lambda_{\tilde{g}, \tilde{A}, \tilde{Z}, m}$ are equal when we choose local trivializations of $S|_{\d M}$ and $\tilde{S}|_{\d M}$ induced by orthonormal frames for $g$ and $\tilde{g}$ respectively.
\end{definition}

The precise meaning of Theorem \ref{boundary determination} is then that if two sets of data $(g,A,Z)$ and $(\tilde{g}, \tilde{A}, \tilde{Z})$ lead to equivalent Dirichlet-to-Neumann maps, then the Taylor series of $g$ and $\tilde{g}$ are equal at the boundary, as are the Taylor series of $Z$ and $\tilde{Z}$, and moreover, after possibly making a gauge transformation of $\tilde{A}$, the Taylor series of $A$ and $\tilde{A}$ are equal at the boundary in any local trivialization where $A_n = 0$ and $\tilde{A}_n = 0$ near the boundary.

The proof of Theorem \ref{boundary determination} follows the recipe of Lee and Uhlmann, which is by now standard in the literature. The idea is to factor the second-order differential operator $\Dir_A^2 + Z - m^2$ into a product of first-order pseudodifferential operators modulo smoothing:
\begin{equation}
    \Dir_A^2 + Z - m^2 \sim \left( D_n + i(E - \theta_n) - i B \right) \left( D_n - i \theta_n + i B \right)
\end{equation}
where $D_n = -i \d_n$, $E = -\frac{1}{2}g^{\a \b} \d_n g_{\a \b}$, $\theta_n = A_n + \omega^{\mathrm{s}}_n$, and $B$ is some pseudodifferential operator to be determined. We can inductively solve for the total symbol of $B$ to show that such a pseudodifferential operator indeed exists. Then, using this factorization and the theory of generalized heat equations, one can show that $B|_{\d M} \sim \Lambda_{g,A,Z,m}$. Finally, we note that the inductive procedure used to determine the symbol of $B$ in terms of the known data $(g,A,Z)$ can be inverted to inductively solve for the normal derivatives of this data at the boundary.

\begin{proof}[Proof of Theorem \ref{boundary determination}]
For this proof, we let $x$ denote the coordinates in a fixed boundary chart for $g$, and write $x = (x', x^n)$, where $x^n$ is the normal coordinate, and $x' = (x^1, \dots, x^{n-1})$ are the tangential coordinates. Greek letters run over $\{1, \dots, n-1 \}$ while Latin indices run over $\{1, \dots, n \}$. So in particular, $(x^\a, 0)$ form coordinates on the boundary. We let $D_k := -i \d_k$. Finally, let $\theta$ be the matrix valued $1$-form representing the connection $\omega^\mathrm{s} \otimes A$ in a local trivialization to be determined, where $\omega^\mathrm{s}$ is the spin connection. By using the Lichnerowicz formula
\begin{equation}
    \Dir_A^2 = \left(\nabla^A\right)^* \nabla^A + \frac{1}{4} R + \frac{1}{2} F_A
\end{equation}
and separating out the normal derivatives from the tangential derivatives, we can write
\begin{align}\label{normal derivatives separated out}
    \Dir_A^2 + Z - m^2 &= D_n^2 + i(E - 2 \theta_n) D_n + Q_2 + Q_1 + Q_0
\end{align}
where $E = -\frac{1}{2} g^{\a \b} \d_n g_{\a \b}$, and where
\begin{align}
    Q_2 &:= -g^{\a \b} \d_a \d_\b, \\
    Q_1 &:= -g^{\a \b} \theta_\a \d_\b + g^{\a \b} \Gamma_{\a \b}^{\g} \d_\g, \\
    Q_0 &:= -g^{ij} \left(\d_i \theta_j\right) - g^{ij} \theta_i \theta_j + g^{ij}, \Gamma^k_{ij} \theta_k + \frac{1}{4} R + \frac{1}{2} F_A + Z - m^2.
\end{align}
Here, $\Gamma^i_{\ jk}$ are the Christoffel symbols of $g$ corresponding to the boundary chart. Note that $Q_i$ is an $i$-th order differential operator involving only tangential derivatives. We want to show that there exists a pseudodifferential operator $B(x,D')$ of order $1$ such that
\begin{equation}\label{factorization}
    \Dir_A^2 + Z - m^2 = \left( D_n + i(E -  \theta_n) - iB \right) \left(D_n - i \theta_n + iB \right) + \sS
\end{equation}
where $\sS$ is a smoothing operator. In writing $B(x,D')$, we mean that $B$ is a map $[0, \epsilon) \to \Psi(\d M)$, which for each value $x_n$ yields a pseudodifferential operator $B(x_n, x', D')$ on the corresponding tangential leaf. In particular we have a well-defined notion of restricting $B$ to the boundary, $B(0, x', D')$. The motivation for this factorization comes from the following Lemma, a proof of which is given following the current proof.

\begin{lemma}\label{B is the DN map}
If $B(x,D')$ exists, then $B|_{\d M} = \Lambda_{g,A,Z,m} + \sS$ where $\sS$ is a smoothing operator.
\end{lemma}

An important implication of the preceding Lemma is that $B(0,x',D')$ has the same total symbol as the Dirichlet-to-Neumann map, and so in doing symbol computations, it suffices to work with the operator $B$ appearing in the factorization.

We now proceed with the proof of Theorem \ref{boundary determination}. In order to show that such a $B(x,D')$ exists, we will use the factorization \eqref{factorization} to determine a formal symbol whose corresponding pseudodifferential operator satisfies \eqref{factorization} in each degree. So, by Equations \eqref{normal derivatives separated out} and \eqref{factorization}, we have
\begin{equation}\label{factorization equals normal derivatives separated out}
    \left( D_n + i(E -  \theta_n) - iB \right) \left(D_n - i \theta_n + iB \right) + \sS = D_n^2 + i(E - 2 \theta_n) D_n + Q_2 + Q_1 + Q_0
\end{equation}
where $\sS$ is a smoothing operator. Expanding the left-hand side of \eqref{factorization equals normal derivatives separated out}, we see it equals
\begin{equation}\label{expansion of factorization}
    D_n^2 + i(E - 2 \theta_n) D_n + i[D_n, B] - \d_n \theta_n + (E - \theta_n) \theta_n - [B,\theta_n] - EB + B^2.
\end{equation}
Replacing the left-hand side of Equation \eqref{factorization equals normal derivatives separated out} with \eqref{expansion of factorization} and rearranging, we get
\begin{equation}\label{simplified equality of factorization and separated}
    i[D_n, B] - [B, \theta_n] - EB + B^2 = Q_2 + Q_1 + Q_0'
\end{equation}
where
\begin{equation}
    Q_0' := Q_0 + \d_n \theta_n - (E-\theta_n) \theta_n.
\end{equation}
We want to consider the total symbols of the left and right-hand sides of \eqref{simplified equality of factorization and separated}. Let $b(x,\x')$ be the symbol of $B(x,D')$, and similarly for $Q_2, Q_1$ and $Q_0'$. Then Equation \eqref{simplified equality of factorization and separated} implies
\begin{equation}\label{the symbol eqn}
    \d_n b + [\theta_n, b] + i \sum_\a \left(\d_{\x_\a} b \right) \left( \d_{x^\a}  \theta_n\right) - Eb + \sum_{\nu} \frac{1}{\nu!} \left( \d_{\xi'}^\n b \right) \left( D_{x'}^\n b \right) = q_2 + q_1 + q_0',
\end{equation}
where $\nu$ runs over multi-indices. Let us assume that $b(x,\x')$ has a formal symbol given by
\begin{equation}\label{formal expansion of b}
    b(x,\x') = \sum_{k \leq 1} b_k(x,\x')
\end{equation}
where $b_k(x,\x')$ is homogeneous in $\x'$ of degree $k$ away from $0$. Plugging \eqref{formal expansion of b} into \eqref{the symbol eqn} and taking the degree $m$ part of the equation, we get
\begin{equation}\label{degree 2 eqn}
    b_1^2 = q_2
\end{equation}
\begin{equation}\label{degree 1 eqn}
    2b_1b_{0} + \d_n b_1 - E b_1 + \sum \left( \d_{\xi'} b_1 \right) \left( D_{x'} b_1 \right) = q_1
\end{equation}
\begin{equation}\label{degree 0 eqn}
    2b_1b_{-1} + \d_n b_0 + [\theta_n, b_0] + i \sum_\a \left( \d_{\x_\a} b_{1} \right)\left( \d_{x^\a} \theta_n \right) - E b_0 + \sum_{\substack{j+k=|\n| \\ 0 \leq j,k \leq 1}} \frac{1}{\nu!} \left( \d_{\xi'}^\n b_j \right) \left( D_{x'}^\n b_k \right) = q_0',
\end{equation}
and
\begin{equation}\label{degree m eqn}
    2b_1b_{m-1} = - \d_n b_m - [\theta_n, b_m] - i \sum_\a \left( \d_{\x_\a} b_{m+1} \right)\left( \d_{x^\a} \theta_n \right) + E b_m - \sum_{\substack{j+k-|\n| = m \\ m \leq j,k \leq 1}} \frac{1}{\nu!} \left( \d_{\xi'}^\n b_j \right) \left( D_{x'}^\n b_k \right)
\end{equation}
for $m \leq -1$. Note that since $q_2(x,\x') = g^{\a \b} \x_\a \x_\b$, Equation \eqref{degree 2 eqn} yields
\begin{equation}\label{principal symbol of b}
    b_1(x,\x') := \pm |\x'|_{g}.
\end{equation}
In particular, we can choose the principal symbol $b_1$ to be a negative scalar, as was needed in the proof of the claim above. It is clear from Equations \eqref{degree 1 eqn}-\eqref{degree m eqn} that for $m < 1$, we can inductively determine $b_{m-1}$ from the first $|m|+1$ terms in \eqref{formal expansion of b}. Therefore, since we have a formal symbol that satisfies \eqref{the symbol eqn}, the corresponding pseudodifferential operator satisfies \eqref{factorization equals normal derivatives separated out}. This completes the proof that the pseudodifferential operator $B$ exists, and hence that the Dirichlet-to-Neumann map is an elliptic pseudodifferential operator of order $1$.

In the other direction, suppose we know the Dirichlet-to-Neumann map $\Lambda_{g,A,Z,m}$ associated to data $(g,A,Z,m)$. Then we know its total symbol in any boundary chart and local trivialization of $E$ over the boundary. By using equations \eqref{degree 2 eqn} -- \eqref{degree m eqn}, we can work backwards and inductively extract the boundary data from the homogeneous terms in the expansion of the symbol. To do this, we fix a boundary chart for $\d M$, and an orthonormal frame $e_\a$ on $\d M$, which we extend into $M$ via parallel transport along $e_n$, the inward pointing normal. This orthonormal frame $e_i$ now induces a local trivialization of the spinor bundle near $\d M$ in which the spin connection takes the form \eqref{spin connection local}. We also pick a local trivialization of $E$.

To start with, the restriction of the metric at the boundary $g|_{\d M}$ is encoded in the function $q_2$ and can therefore be recovered from the principal symbol $b_1$ as per equation \eqref{degree 2 eqn}. Upon simplifying and rearranging, Equation \eqref{degree 1 eqn} yields
\begin{equation}\label{degree 1 step of recovery}
    b_0 = -\frac{i}{2} g^{\a \b} \theta_{\a} \frac{\x_\b}{|\x|} + \frac{1}{2} \left( Eg^{\a \b} - \frac{1}{2} \d_n g^{\a \b} \right)\frac{\x_\a \x_\b}{|\x|^2} + F_0\left(g_{\a \b}|_{\d M} \right)
\end{equation}
for all $\x \in T^*\d M \setminus \{0\}$, where $F_1$ is some function of its arguments involving only tangential derivatives along $\d M$. In particular, $F_1$ is known once $g|_{\d M}$ is known. Thus, from $b_0$ we can determine the first two terms on the right-hand-side of \eqref{degree 1 step of recovery}. Moreover, since these first two terms have opposite parity with respect to $\x$, we can determine each term separately by substituting different values for $\x$. Therefore, knowing $g^{\a \b}|_{\d M}$, we can recover
\begin{equation}\label{theta is the sum of two connections}
    \theta_\a|_{\d M} = \id_S \otimes A_\a|_{\d M} - \frac{1}{4} \omega^i_{\ j}(\d_\a)|_{\d M} \gamma(e_i) \gamma(e_j) \otimes \id_E
\end{equation}
where $\omega^i_{\ j}(\d_\a)$ denotes the connection $1$-form for the Levi-Civita connection with respect to the orthonormal frame $e_i$, evaluated on the coordinate vector $\d_\a$. By the Clifford relations, the matrices $\gamma(e_i)\gamma(e_j)$ occurring in the second term of \eqref{theta is the sum of two connections} are traceless. We can therefore recover $A_\a$ as a partial trace over the spinor bundle: 
\begin{equation}
    A_{\a}(0,x') = \frac{1}{\rank{S}} \Tr_S \theta_\a|_{\d M}.
\end{equation}
Since $A_n = 0$ in this gauge by assumption, we have recovered $A|_{\d M}$. Knowing $A|_{\d M}$, from Equation \eqref{theta is the sum of two connections} we can recover the matrix
\begin{equation}
    -\frac{1}{4} \omega^i_{\ j}(\d_\a)|_{\d M} \gamma(e_i) \gamma(e_j).
\end{equation}
Using the fact that $\gamma(e_i) \gamma(e_j)$ provide an orthonormal set of matrices with respect to the trace inner product, we can extract $\omega^i_{\ j}(\d_\a)|_{\d M}$. The final obstacle is to obtain the Christoffel symbols of $g$ with respect to the boundary normal chart. For this, let $h$ be the matrix defined by $\d_\a = h^\b_{\ \a} e_\b$. Then, if $\Gamma^{i}_{\ jk}$ denotes the connection coefficients in the boundary normal chart, we have
\begin{equation}\label{connection coefficient transformation law}
    \Gamma^{i}_{\ \a j} = \left( h^{-1} \right)^i_{\ k} \left( \omega^k_{\ \l}(\d_\a) \right) h^{\l}_{\ j}  + \left( h^{-1} \right)^i_{\ k} \d_{\a} h^k_{\ j}.
\end{equation}
Since $h$ is known on the boundary, so is $\omega^i_{\ j}(\d_\a)|_{\d M}$, and we can recover $\Gamma^i_{\a j}$. In particular, we recover the normal derivative of the metric at the boundary, since
\begin{equation}
    \Gamma^n_{\a \b} = -\frac{1}{2} \d_n g_{\a \b}.
\end{equation}
Thus we have recovered $A_\a|_{\d M}$ $\d_n g|_{\d M}$. The next equation \eqref{degree 0 eqn} leads to an equation of the form
\begin{align}\label{degree -1 step of recovery}
     b_{-1} &= -\frac{i}{2} g^{\a \b} \d_n \theta_\a \frac{\x_\b}{|\x|^2} + \left( T_{-1}^{\a \b}\left(g|_{\d M}, \d_n g|_{\d M}, \d_n^2 g|_{\d M}, \d_n A|_{\d M} \right) + Z|_{\d M} g^{\a \b} \right) \frac{\x_\a \x_\b}{|\x|^3} \\
     &\ \ \ \ \ \ + F_{-1}\left( g|_{\d M}, \d_n g |_{\d M}, A|_{\d M} \right).
\end{align}
So as before, we can recover $\d_n \theta_\a|_{\d M}$. Moreover, since the local trivialization of $S$ is induced by an orthonormal frame, the matrices $\gamma(e_i)$ are constant, and thus $\d_n \theta_\a$ takes the form
\begin{equation}
    \d_n \theta_\a = - \frac{1}{4} \d_n\left( \omega^i_{\ j}(\d_\a) \right) \g(e_i) \g(e_j) \otimes \id_E + \id_S \otimes \d_n A_\a.
\end{equation}
Taking a partial trace over $S$ as before, we can determine $\d_n A_\a|_{\d M}$, and therefore also $\d_n\left( \omega^i_{\ j}(\d_\a) \right)$. Using the fact that $e_\b$ is parallel along $\d_n$, and the fact that we have determined $\Gamma^i_{\ jk}|_{\d M}$, it follows that $\d_n h |_{\d M}$ is known. Therefore, taking the normal derivative of the transformation law \eqref{connection coefficient transformation law} for the connection coefficients, we get
\begin{equation}
    \d_n \Gamma^i_{\ \a j} = \left( h^{-1} \right)^i_{\ k} \d_n \left( \omega^k_{\ \l}(\d_\a) \right) h^{\l}_{\ j}  + S\left(h|_{\d M}, \d_n h|_{\d M}, g|_{\d M}, \d_n g|_{\d M} \right),
\end{equation}
where $S$ is a function of known quantities. We can therefore recover the normal derivatives of the Christoffel symbols of $g$ in the boundary normal chart, and in particular, we can recover
\begin{equation}
    \d_n \Gamma^n_{\ \a \b} |_{\d M} = -\frac{1}{2} \d_n^2 g_{\a \b} |_{\d M}. 
\end{equation}
Thus, having recovered the second derivatives of the metric at $\d M$, and the first derivatives of the connection at $\d M$, the only unknown remaining is the endomorphism $Z|_{\d M}$ in the even term of \eqref{degree -1 step of recovery}, which is now easily recovered.

Continuing in this fashion, at step $m$ we get an equation of the form
\begin{align}
    b_{1-m} &= -\frac{i}{2} g^{\a \b} \d^{m-1}_n\theta_\a \frac{\x_{\b}}{|\x|^m} + \left( T_{1-m}^{\a \b} + \d_n^{m-2} Z|_{\d M} \right) \frac{\x_\a \x_\b}{|\x|^{m+1}} \nonumber \\
    &\ \ \  + F_{1-m}\left( g_{\a \b}|_{\d M}, \dots, \d^{m-1}_n g_{\a \b} |_{\d M}, A_{\a}|_{\d M}, \dots, \d_n^{m-2} A_{\a} |_{\d M}, Z|_{\d M}, \dots, \d_n^{m-3} Z|_{\d M}  \right),
\end{align}
where $T_{1-m}^{\a \b}$ is a known quantity depending on derivatives of $g$ up to order $m$, derivatives of $A$ up to order $m-1$, and derivatives of $Z$ up to order $m-3$. Therefore, as before we can recover $\d_n^{m-1} \theta_\a$, and by taking traces, $\d_n^{m-1} A_\a$ and $\d_n^{m-1} \omega^i_{\ j}(\d_\a)$. Then, taking derivatives of the transformation law, we can recover $\d_n^{m-1}\Gamma^n_{\ \a \b}$ as above. In this fashion we recover the normal derivatives of $A$ and $g$ at the boundary.
\end{proof}

We now give a proof of Lemma \ref{B is the DN map}, which is essentially identical to the proof given in \cite{Lee1989}; we nonetheless include it here in the interest of being self-contained.

\begin{proof}[Proof of Lemma \ref{B is the DN map}]
Let us suppose that a first-order pseudodifferential operator $B(x,D')$ exists that satisfies Equation \eqref{factorization}. Let $\chi \in H^\frac{1}{2}(S|_{\d M})$, and let $\vp \in \sD'(S)$ be the solution to the Dirichlet problem \eqref{Dirichlet problem for chi}. Since $(\Dir_A^2 + Z - m^2) \vp = 0$, the factorization \eqref{factorization} yields the following system of equations:
\begin{align}
    \psi :&= \left( D_n - i \theta_n + iB \right) \vp, \  \ \vp|_{x^n = 0} = \chi \label{factor equation for phi} \\
    h &= \left( D_n + i(E - \theta_n) -iB \right) \psi \label{factor equation for psi}
\end{align}
where $h$ is some smooth spinor field near $\d M$ given by $h := -\sS \psi$. Note that although the regularity of $\psi$ is not a priori known, we know that $h$ is smooth since $\sS$ is a smoothing operator. Writing $t := T - x^n$, we can write Equation \eqref{factor equation for psi} as
\begin{equation}\label{backwards heat equation}
    \d_t \psi - (B - E + \theta_n) \psi = -ih.
\end{equation}
Equation \eqref{backwards heat equation} is a generalized backwards heat equation. Now, by elliptic regularity, we know that $\vp$ is smooth in the interior, and therefore so is $\psi$ by equation \eqref{factor equation for phi}. In particular, $\psi|_{x^n = T}$ is smooth. We'll see later that we can choose the principal symbol of $B$ to be a negative scalar, which implies that the backwards heat equation \eqref{backwards heat equation} with initial condition $\psi|_{x_n = T}$ is well-posed. Thus the solution operator is a smoothing operator, and since $\psi|_{x_n = T}$ is smooth, so is $\psi$. In particular, $\psi|_{x^n = 0}$ is smooth. So, let us define an operator $R$ by
\begin{equation}\label{R the smoothing operator}
    R \chi := \psi|_{\d M}.
\end{equation}
Then we have that $R$ is smoothing by construction, and moreover
\begin{align}
    R \chi := \psi|_{\d M} &= \left( \left( D_n - i\theta_n + i B(x,D') \right) \vp \right)\big|_{\d M} \nonumber \\
    &= -i\left( \nabla_n^A \vp\right) \big|_{\d M} + i B(0,x',D')\chi \nonumber \\
    &= -i\Lambda_{g,A,Z,m} \chi + iB(0,x',D')\chi. \label{R equation for DN map}
\end{align}
So indeed, we have $B|_{\d M} = \Lambda_{g,A,Z,m}$ modulo smoothing.
\end{proof}

\begin{remark}\label{conformal remark}
The proof of Theorem \ref{boundary determination}, unlike in the case of the scalar or connection Laplacian, holds for $\dim{M} = 2$. Moreover, it is unnecessary to normalize the metric in order to obtain the Taylor series of the endomorphism $Z$ at the boundary, as is done for the connection Laplacian in \cite{Ceki2020}. These observations can be attributed to the fact that conformal covariance of the Dirac operator does not extend to even powers thereof; see \cite{Fischmann2014} for details.
\end{remark}

\section{Recovering real-analytic Yang--Mills--Dirac connections}\label{section: recovering connections}

From here on, we fix a background metric $g$ on a compact spin manifold $M$ with boundary $\d M$, and let $S$ be the spinor bundle of $M$. Let $E$ be a Hermitian vector bundle of rank $N$ over $M$, associated to its principal $U(N)$-bundle $P$ of unitary frames, and let $A$ be a $U(N)$-connection on $P$, whose spinorial Dirichlet-to-Neumann map $\Lambda_A$ is given. We want to introduce an auxiliary $E$-valued spinor $\phi$ to which $A$ couples in a physically interesting way, and whose boundary value is known, and investigate the extent to which $\Lambda_A$ determines the connection $A$ modulo the action of the gauge group. To this end, we introduce the Yang--Mills--Dirac equations:
\begin{equation}\label{YMD spinor}
    \Dir_A \phi = m \phi
\end{equation}
\begin{equation}\label{YMD connection}
    d_A^*F_A = J(\phi),
\end{equation}
where $m \in \R$, and $J : \Gamma(S \otimes E) \to \Lambda^1(M, \ad{P})$ associates to each $E$-valued spinor $\vp$ its {\em current}, and is defined as follows: for $\vp \in \Gamma(S \otimes E)$, there is a unique $J(\vp) \in \Lambda^1(M,\ad{P})$ such that
\begin{equation}
    \< \vp, \x \cdot \vp  \> = \< J(\vp), \xi \>
\end{equation}
for all $\x \in \Lambda^1(M,\ad{P})$. Note that for any gauge transformation $G$, we have $J(G \vp) = G J(\vp) G^{-1}$. The physical significance of Equations \eqref{YMD spinor}--\eqref{YMD connection} is that they represent a gauge field $A$ interacting with some matter field $\phi$ of mass $m$, by means of its charge current $J(\phi)$.  For details on the Yang--Mills--Dirac system, we refer the reader to \cite{Bleecker2005}.

Because we are concerned primarily with the Dirichlet-to-Neumann map of the second-order operator $\Dir_A^2$, we will assume that the connection $A$ and the auxiliary spinor field $\phi$ satisfy the more general second-order system:
\begin{equation}\label{YMD-2 spinor}
    \Dir_A^2 \phi = m^2 \phi
\end{equation}
\begin{equation}\label{YMD-2 connection}
    d_A^*F_A = J(\phi).
\end{equation}
Note that every solution of \eqref{YMD spinor}--\eqref{YMD connection} is a solution of \eqref{YMD-2 spinor}--\eqref{YMD-2 connection}, but not vice versa.

In this section, we prove the following result:

\begin{theorem}\label{main thm}
Let $A$ and $B$ be $U(N)$-connections as above, real-analytic in the interior, and let $\phi$ and $\psi$ be $E$-valued spinors on $M$ such that $(A, \phi)$ and $(B,\psi)$ both satisfy the second-order Yang--Mills--Dirac system \eqref{YMD-2 spinor}--\eqref{YMD-2 connection}, and such that $\phi|_{\d M} = \psi|_{\d M}$. If $\Lambda_{A,m} = \Lambda_{B,m}$, then $A$ and $B$ are locally gauge equivalent about any point in $M$.
\end{theorem}

\begin{proof}
    We work in a real-analytic trivialization over an open set $U$ intersecting the boundary $\d M$. We can find a smooth global gauge transformation $F$ satisfying
    \begin{equation}
    \begin{cases}
     \d_n F = - A_n F \ \text{in\ } U \\
     F|_{\d M} = \id,
    \end{cases}
\end{equation}
% why is F unitary in this case? see the theorem in Ckc paper 
which is moreover real-analytic in $U$, since $A_n$ is real-analytic there. By applying this gauge transformation $F$, we get a pair $(A', \phi')$ that is real-analytic in $U$, such that $A_n' = 0$ in the local trivialization over $U$. Moreover, since $F|_{\d M} = \id$, it follows that $\Lambda_{A'} = \Lambda_A$. Therefore, doing the same for $B$, we may assume without loss of generality that $A$ and $B$ are smooth connections, satisfying $A_n = B_n = 0$ in the local trivialization over $U$, and real-analytic there. 

Since $\Lambda_A = \Lambda_B$, Theorem \ref{boundary determination}, implies that in this local trivialization we have
\begin{equation}\label{normal derivatives of A and B}
    \d_n^k(A - B) \big|_{\d M}= 0, \ \ k \geq 0.
\end{equation}
That is, $A$ and $B$ have the same Taylor series at the boundary in this local trivialization.

We want to extend this observation to $\phi$ and $\psi$. Note that by assumption $\phi|_{\d M} = \psi|_{\d M}$. Using this fact, as well as the fact that $A_n = B_n = 0$ in $U$, the equality $\Lambda_A = \Lambda_B$ yields
\begin{equation}
    \d_n (\phi - \psi)|_{\d M} = \left( \nabla^A_n \phi - \nabla^B_n \psi \right)|_{\d M} = \Lambda_A\left( \phi|_{\d M} \right) - \Lambda_B\left( \psi|_{\d M} \right) = 0.
\end{equation}
So $\d_n(\phi - \psi)|_{\d M} = 0$ in this local trivialization. Moreover, \eqref{YMD-2 spinor} yields
\begin{equation}
    \d_n^2(\phi - \psi)|_{\d M} = \left( \Dir_A^2 \phi - \Dir_B^2 \psi \right)|_{\d M} - \left( \phi - \psi \right)|_{\d M} = 0,
\end{equation}
where we again have used $\d_n (\phi - \psi)|_{\d M} = 0$, $(\phi - \psi)|_{\d M} = 0$, and $\d_n^k(A - B)|_{\d M} = 0$ for all $k$. By taking derivatives of \eqref{YMD-2 spinor}, similar arguments yield $\d_n^k(\phi - \psi)|_{\d M} = 0$.

Now we want to make a local gauge transformation of $A$ in $U$ so that the new connection $A'$ satisfies $d^* A' = 0$. The key ingredient in this step is the Cauchy--Kovalevskaya theorem; this is one place where the real-analyticity hypothesis plays a crucial role. To this end, we recall the following well-known result; see \cite[Theorem 5.4]{Hall2003}, for example:

\begin{lemma}\label{lem d(e^S)}
Let $\fg$ be a matrix Lie algebra. Then for $S : \R^n \to \fg$, we have
\begin{equation}\label{d(e^S)}
    d(e^S) = e^S \left( \frac{1 - e^{-\ad{S}}}{\ad{S}} \right)(dS)
\end{equation}
where $\ad{S} \in \End{\fg}$ is the endomorphism $X \mapsto [S,X]$.
\end{lemma}

For $S \in \fu(N)$, let us denote the endomorphism in brackets acting on $dS$ in \eqref{d(e^S)} by $\Theta(S)$. Note that $\Theta(S)$ is defined by a power series in $\ad{S}$ and is therefore real-analytic in $S$. Moreover, note that $\Theta(0) = \id$, and therefore $\Theta(S)$ is invertible for small values of $S$.

Now, let us consider the following Cauchy problem for $S : U \to \fu(N)$,
\begin{equation}\label{CK system 1}
    \begin{cases}
    d^*\left(e^{-S} A e^S + e^{-S} d(e^S)\right) = 0, \\
    S|_{U \cap \d M} = 0, \\
    \d_n S |_{U \cap \d M} = 0.
    \end{cases}
\end{equation}
Using Lemma \ref{lem d(e^S)}, Equation \eqref{CK system 1} can be written in the form
\begin{equation}\label{CK system}
    \begin{cases}
    -g^{ij} \nabla_i \nabla_j S = \Theta(S)^{-1} F(x, S, dS ; A) \\
    S|_{U \cap \d M} = 0, \\
    \d_n S |_{U \cap \d M} = 0,
    \end{cases}
\end{equation}
where 
\begin{align}\label{expression for F}
    F(x,S,dS; A) &= g^{ij} e^{-S} \Theta(-S)(\d_i S) A_j e^S + g^{ij} e^{-S} \left(\nabla_i A_j \right) e^S + g^{ij} e^{-S} A_j e^S \Theta(S)(\d_i S)  \nonumber \\
    &\ \ \ \ \ \ \ \ \ \ \ \ \ \ \ \ \ \ \ \ \ \ \ \ \ \ \ \ \ \ \ \ \ \ \ \ \ \ \ \ \ \ \ \ \ \ \ \ \ \ \ \ \ \ \ \ \ \ \ \ \ \ \ \ \ \ \ \ + g^{ij}(D \Theta)(S)(\d_i S, \d_j S). 
\end{align}
Note that since $S \mapsto \Theta(S)$ is real-analytic near $S = 0$, as are the connection $A$ and metric $g$, the function $F(x,S,dS;A)$ is real-analytic in a neighbourhood of $(x, 0, 0; A)$ for $x \in U \cap \d M$. Therefore, since $\Theta(S)^{-1}$ is well-defined and real-analytic in a neighbourhood of $S = 0$, the Cauchy--Kovalevskaya theorem (see \cite[\S 16.4]{Taylor2011} for example) yields, after possibly shrinking $U$, the existence of a real-analytic solution $S : U \to \fu(N)$ to \eqref{CK system}. 

Therefore, $e^S$ yields a local gauge transformation over $U$. Applying this gauge transformation to $A$, we get a new connection $A' := S^{-1} A S + S^{-1} dS$ over $U$, which satisfies $d^* A' = 0$. Moreover, since the Yang--Mills--Dirac equations \eqref{YMD-2 spinor}--\eqref{YMD-2 connection} are gauge-invariant, the pair $(A', \phi')$, where $\phi' := e^{-S} \phi$, continues to satisfy the Yang--Mills--Dirac system. In particular, the pair $(A',\phi')$ satisfies the following nonlinear elliptic system over $U$,
\begin{equation}\label{elliptic system}
    \begin{cases}
    \Dir_{A'}^2 \phi' = m^2 \phi', \\
    d_{A'}^* F_{A'} = J(\phi'), \\
    dd^*A' = 0.
    \end{cases}
\end{equation}
In the same manner, we may conclude, after possibly shrinking $U$ again, that there exists a function $T : U \to \fu(N)$ such that the connection $B'$ defined by $B' := T^{-1} B T + T^{-1} dT$, and the spinor $\psi' := e^{-T} \psi$, also satisfy the nonlinear elliptic system \eqref{elliptic system} over $U$.

We now want to show that $S$ and $T$ have the same Taylor series at the boundary. To this end, note that since $\Theta(0) = \id$, Equation \eqref{expression for F} yields $\Theta^{-1}(0)F(x, 0, 0; A) = d^*A(x)$ for any $x \in \d M$. Furthermore, since $S|_{\d M} = T|_{\d M} = 0$ and $\d_n S |_{\d M} = \d_n T|_{\d M} = 0$, it follows that $dS|_{\d M} = dT|_{\d M} = 0$. These observations, along with \eqref{CK system} give us
\begin{equation}
    \d_n^2 (S - T)|_{\d M} = \left( \Delta S - \Delta T \right)|_{\d M} = \left( F(\cdot, 0, 0; A) - F(\cdot, 0, 0' B) \right)|_{\d M} = \left( d^*A - d^*B \right)|_{\d M} = 0,  
\end{equation}
the last equality holding since $A$ and $B$ have the same Taylor series at the boundary. By taking derivatives of \eqref{CK system} and using $\d_n^k (A - B)|_{\d M} = 0$ for all $k$, we get $\d_n^k(T - S) = 0$ for all $k$. Therefore, it follows that $\d_n^k(A' - B')|_{\d M} = 0$ and $\d_n^k(\phi' - \psi')|_{\d M} = 0$ for all $k \geq 0$. Since $(A',\phi')$ and $(B',\psi')$ both satisfy \eqref{elliptic system} in $U$, and have equal Taylor series at the boundary, it follows from the unique continuation principle for elliptic systems with scalar principal part (see \cite[Theorem 3.5.2]{Isakov2017} for example) that $(A',\phi') = (B',\psi')$ in $U$. In particular, this shows that $(A,\phi)$ is gauge equivalent to $(B, \psi)$ over $U$, by a gauge transformation that is equal to the identity on $\d M$.

We have so far shown that the original pairs $(A,\phi)$ and $(B,\psi)$ are locally gauge equivalent near the boundary $\d M$ by a real-analytic gauge transformation. We now want to extend this to the interior of $M$. For this final step, we use the real-analyticity of $A$ and $B$ in the interior as follows. 

First, we have the following lemma from \cite{Ceki2020}, which is stated there for the case $\cL = d_A^*d_A$ on some Hermitian vector bundle $E$. The proof easily extends to the case $\cL = \Dir_A^2 - m^2$ presently under consideration, provided of course that $m^2 \notin \Spec{\Dir_A^2}.$ 

\begin{lemma}\label{cekic paths}
Let $\beta \sub M^{\mathrm{int}}$ be an embedding of $[0,1]$ into $M$. Then there exists smooth sections $\vp_1, \dots, \vp_m$, harmonic with respect to $\cL$, and having $\supp{\left( \vp_i|_{\d M} \right)} \sub \Gamma$ for some given non-empty open subset $\Gamma \sub \d M$, that form a frame for $E \otimes S$ at every point along $\beta$. 
\end{lemma}

\begin{remark}
The proof involves solving the Dirichlet problem for $\cL$, and then extending the resulting sections to global ones. See \cite[Lemma 6.1]{Ceki2020} for details.
\end{remark}

Armed with Lemma \ref{cekic paths}, we fix a point $x \in M$, and let $\beta$ be a path from $x$ to a point $y$ near $\d M$ such that $A$ and $B$ are locally gauge equivalent near $y$. We consider a tubular neighbourhood $W$ of $\beta$, over which $S \otimes E$ is trivial. Since $A$ is real-analytic, the harmonic sections $\vp_i$ given by Lemma \ref{cekic paths} are real-analytic. Therefore, they provide a real-analytic trivialization of $S \otimes E$ over $U$.

Now let $\psi_i$ be the solution to
\begin{equation} \label{dirichlet problem for frame}
    \begin{cases}
            \Dir_B^2 \psi_i = m^2 \psi_i \\
            \psi_i|_{\d M} = \vp_i|_{\d M}.
    \end{cases}
\end{equation}
Since $B$ is real-analytic, so are the $\psi_i$. Let us construct the endomorphism
\begin{equation}\label{defn of h}
    H_0 = \sum_j \psi_j \otimes \vp_j^{\flat}
\end{equation}
where $\vp_j^{\flat}$ indicates the dual frame. Then $H_0$ is real-analytic in $W$. Moreover, near $y$, we know that there exists a real-analytic gauge transformation $G$ taking $A$ to $B$, satisfying $G|_{\d M} = \id$. Let $G_0 := \id \otimes G$. Then we have that $G_0^{-1} \vp_i = \psi_i$ near $\d M$. Indeed, both $G_0^{-1} \vp_i$ and $\psi_i$ satisfy the same elliptic equation,
\begin{equation}
    \Dir_B^2 \left(G_0^{-1} \vp_i \right) = G_0^{-1} \Dir_A^2 \vp_i = m^2 G_0^{-1} \vp_i
\end{equation}
and boundary condition $G_0^{-1} \vp_i|_{\d M} = \psi_i|_{\d M}$. Also, by equality of the Dirichlet-to-Neumann maps, 
\begin{equation}
    \nabla^B_n\left(G_0^{-1} \vp_i\right)|_{\d M} = G_0^{-1} \nabla^A_n \vp_i |_{\d M} = \nabla^B \psi_i |_{\d M}.  
\end{equation}
So again, by unique continuation, we have that $G_0^{-1} \vp_i = \psi_i$ near $\d M$. This is precisely the property of $H$ defined in \eqref{defn of h}. Therefore $H = G_0$ near $\d M$. In particular, taking partial traces, we find that
\begin{equation}\label{partial trace of H}
    G = \frac{1}{\rank{S}} \Tr_S{H_0}.
\end{equation}
Since $G$ is unitary near $\d M$, it follows from real-analyticity that the right-hand of \eqref{partial trace of H} is unitary in $W$. Denoting this right-hand-side by $H$, it follows that $H$ is a unitary automorphism of $E$ over $W$. Moreover, since $B = H^{-1}AH + H^{-1} dH$ near $\d M$, it follows again from real-analyticity that this holds over $W$. This proves that $A$ and $B$ are gauge equivalent about any point in $M$, and are thus locally gauge equivalent.
\end{proof}

\begin{corollary}\label{corollary abelian}
Let $A$ and $B$ be as above. If $rank{E} = 1$, then $A$ and $B$ are globally gauge equivalent via a gauge transformation $G$ satisfying $G|_{\d M} = 1$.
\end{corollary}

\begin{proof}
In the case $N = 1$, global gauge transformations are elements of $C^\8(M, S^1)$. The same arguments use in preceding proof show that $(A,\phi)$ is locally gauge equivalent to $(B, \psi)$ everywhere in $M$. That is, about every point in $M$ there exists a locally defined smooth $S^1$-valued function satisfying $f \phi = \psi$. Since $\phi$ and $\psi$ are solutions to elliptic equations, the set on which they do not vanish is dense. It thus follows that one can patch these local $S^1$-valued functions to obtain a global well-defined $S^1$-valued function, which is by definition a gauge transformation from $A$ to $B$, and equal to $1$ on $\d M$.
\end{proof}

\begin{remark}
The ease with which we obtain the preceding corollary can be attributed to the fact that Theorem \ref{main theorem} yields not only a gauge transformation between the connections, but also between the spinors, on which the gauge action takes a particularly simple form when $N = 1$.
\end{remark}

Here, as well as in the proof of Theorem \ref{main theorem}, we have seen how the introduction of a spinor field coupled to a unitary connection affects the recovery of the connection up to gauge from the Dirichlet-to-Neumann map of its twisted Dirac Laplacian. This motivates the following question: in what ways does the introduction of a spin structure affect the recovery of other geometric structures from boundary data associated to a Dirac Laplacian. For instance, regarding the Calder\'{o}n problem for the metric, it would be worthwhile to study if the spin structure can be exploited to facilitate the recovery of the metric up to isometry, even in dimension $2$, from the Dirichlet-to-Neumann map of the Dirac Laplacian.

\section*{Acknowledgements}

The author wishes to acknowledge financial support from the Natural Sciences and Engineering Research Council of Canada in the form of a Canada Graduate Scholarship for Doctoral studies, as well as Professor Niky Kamran for his helpful suggestions and comments.

\end{document}